\documentclass[12pt,a4paper,twoside,final]{article}
\usepackage[utf8]{inputenc}
\usepackage[T1]{fontenc}
\usepackage{amssymb, amscd, amsthm, amsmath,latexsym, amstext,mathrsfs,multicol,CJK, mathtools}
\usepackage[all]{xy}
\usepackage[dvips]{graphicx}
\usepackage{physics}
\usepackage{physics2}
\usephysicsmodule{ab}

\usepackage[colorlinks, linkcolor=blue, anchorcolor=green, citecolor=red, backref=page]{hyperref}

\usepackage{color}
\usepackage{graphicx}
\usepackage{dcolumn}
\usepackage{bm}
\usepackage{enumerate}
\usepackage{amsfonts}

\usepackage{amsmath,comment}
\usepackage{mathrsfs}
\usepackage[a4paper,top=2cm,bottom=2cm,left=2cm,right=2cm,marginparwidth=1.75cm]{geometry}

\usepackage{enumitem} 
\theoremstyle{definition}
\newtheorem{thm}{Theorem}[section]
\newtheorem{coro}[thm]{Corollary}
\newtheorem{lemma}[thm]{Lemma}
\newtheorem{prop}[thm]{Proposition}

\newtheorem{remark}[thm]{Remark}
\newtheorem{defn}[thm]{Definition}
\newtheorem{example}[thm]{Example}

\newcommand{\Gal}{{\rm Gal}}
\newcommand{\Pic}{{\rm Pic}}

\newcommand{\Br}{{\rm Br}}

\newcommand{\Spec}{{\rm Spec}}
\newcommand{\car}{\mathrm{char}}

\newcommand{\rO}{\mathrm{O}}

\newcommand{\bfA}{\mathbf{A}}

\newcommand{\calO}{\mathcal{O}}
\newcommand{\calG}{\mathcal{G}}
\newcommand{\calH}{\mathcal{H}}
\newcommand{\calX}{\mathcal{X}}
\newcommand{\St}{\mathrm{St}}
\newcommand{\rH}{\mathrm{H}}
\newcommand{\rHet}{\mathrm{H}_{\mathrm{\acute{e}t}}}
\newcommand{\Ker}{\mathrm{Ker}}
\renewcommand{\d}{\mathrm{d}}

\newcommand{\GL}{\mathrm{GL}}
\newcommand{\bF}{\mathbb{F}}
\newcommand{\laurent}[1]{(\!(#1)\!)}
\newcommand{\bA}{\mathbb{A}}
\newcommand{\Spin}{\mathrm{Spin}}
\newcommand{\Mod}[1]{\ (\mathrm{mod}\ #1)}
\newcommand{\vol}{\mathrm{vol}}

\newcommand{\ord}{\mathrm{ord}}
\DeclareFontFamily{U}{wncy}{}
\DeclareFontShape{U}{wncy}{m}{n}{%
	<5>wncyr5%
		<6>wncyr6%
	<7>wncyr7%
	<8>wncyr8%
	<9>wncyr9%
	<10>wncyr10%
	<11>wncyr10%
	<12>wncyr6%
	<14>wncyr7%
	<17>wncyr8%
	<20>wncyr10%
	<25>wncyr10}{}
\DeclareMathAlphabet{\cyrille}{U}{wncy}{m}{n}
\def\Sha{\cyrille X}

\usepackage{etoolbox}
\apptocmd{\theliography}{\setlength{\itemsep}{0pt}}{}{}

\begin{document}
\title{Counting integral points in homogeneous spaces over function fields}

\author{Sheng Chen and Jing Liu\footnote{Corresponding author.}}
\date{}

\maketitle
\begin{abstract}We establish the asymptotic formula for the number of integral points in non-compact symmetric homogeneous spaces of semi-simple simply connected algebraic groups over global function fields, given by the sum of the products of local densities twisted by suitable Brauer elements. As an application, we count the number of representations of a scalar by a sum of at least three squares.
\end{abstract}

MSC 2020: 11D45, 11G35, 14G05, 14G12, 14M17


Key words: Brauer--Manin obstruction, strong approximation, integral points, Tamagawa measure, homogeneous spaces.

\section{Introduction} \label{sec-introduction}

The Hardy--Littlewood circle method is classically used for counting integral points on varieties. But for this method to be applicable, one necessary condition is that the variety in question satisfies the Hasse principle. Aiming to go beyond this limitation, Borovoi and Rudnick \cite{Borovoi-Rudnick95} introduced the so-called class of Hardy--Littlewood varieties, for which the integral Hasse principle might not hold, yet still the difference between the asymptotic behavior of their integral points and the Hardy--Littlewood expectation could be described in terms of some density function on the adelic space. In particular, they showed that certain homogeneous spaces of semi-simple groups over number fields are Hardy--Littlewood. These density functions, however, were described using the Kottwitz invariant and not determined locally. Relating this counting problem to strong approximation with Brauer--Manin obstruction, Wei and Xu \cite{Wei-Xu16} reinterpreted the asymptotic formula for integral points as the sum of the products of local densities twisted by Brauer elements. This enabled them to carry out explicit computations, e.g., for the variety of matrices with a fixed characteristic polynomial.  

In this note we complement the work of Wei and Xu by establishing the corresponding asymptotic formula for symmetric homogeneous spaces over global function fields (see Theorem \ref{th1} and \ref{th2}). We also apply this result to the concrete situation where the homogeneous space is given by the representation of an element by a sum of squares. Our result rests on the work by Demarche and Harari \cite{Demarche-Harari22} on strong approximation over global function fields and the equidistribution property developed in \cite{Benoist-Oh12}. 

\subsection*{Notations and conventions}

Let $F$ be a
global function field of characteristic $p$ with constant field $\mathbb{F}_q$. We denote by $\eta_F$ the genus of $F$. Let $\Omega_F$ be the set of all places of $F$. For a place $v\in \Omega_F$, let $F_v$ be the $v$-adic completion of $F$ with ring of integers $\calO_v$. For any
finite subset $S$ of $\Omega_F$, the ring of
$S$-integers is the collection of elements in $F$ which are integral
outside $S$ and is denoted by $\calO_S$. Denote the ad\`ele ring
of $F$ by $\bfA_F$ and we also put $F_S\coloneqq\prod_{v\in S}F_v$.

Let $\overline{F}$ be a fixed separable closure of $F$. For any variety $Z$ over $F$, we write $\overline{Z}=Z\times_F \overline{F}$, and set
\[
\Br(Z)=\rHet^2(Z, \mathbb G_m),  \ \ \ \Br_1(Z)= \Ker[\Br(Z)\rightarrow \Br(\overline{Z})].
\] 
Let $G$ be a connected reductive linear algebraic group over $F$ and $\calX$
be a separated scheme of finite type over $\calO_S$ whose generic fiber
\[
X={\calX} \times_{\calO_S} F  \simeq H \backslash G
\] is a right homogeneous space of $G$ admitting a point $P\in X(F)$ with $H$ as its stabilizer.
The rational point $P$ determines a map $\pi\colon G\to X, \ g\mapsto P\cdot g,$ which induces the following commutative diagram \[
\xymatrix{
\Br(X) \ar[d] \ar[r]^{\pi^*} & \Br(G) \ar[d]\\
\Br(\overline{X}) \ar[r]^{\overline{\pi}^*} & \Br(\overline{G}).
}
\]

Set \[
\Br_1(X, G)=\Ker[\Br(X)\to \Br(\overline{G})].
\]Then $\Br_1(X,G)$ is a subgroup of $\Br(X)$ containing $\Br_1(X)$.
Following \cite{Demarche-Harari22}, we further consider the subgroup $\Br_{1,P}(X, G)$ of $\Br_1(X, G)$ consisting of those elements $\alpha $ such that $\alpha(P)=0$. Quotienting by $\Br(F)$ induces an isomorphism $\Br_{1,P}(X, G)\simeq \Br_1(X, G)/\Br(F)$. 


The obvious necessary condition
for ${\calX}(\calO_S)\neq \emptyset$ is
\begin{equation}\label{loc} \prod_{v\in (\Omega_F \setminus  S)} {\calX}(\calO_v)\neq \emptyset, \end{equation} which is
assumed throughout this paper. By separatedness of $\calX$, one can naturally regard $\calX(\calO_v)$
as an open and compact subset of $X(F_v)$ with $v$-adic
topology for $v\notin S$. We assume that $X$ is affine and fix some closed immersion
\begin{equation}\label{coordi} X \hookrightarrow  \Spec(F[x_1, \cdots, x_n]) . \end{equation}  By \cite[Theorem~A]{Richardson77}, $X$ being affine is equivalent to $H$ being reductive. For simplicity, we will also assume that $H$ is connected. 

For any $F$-algebra $E$, a point in $X(E)$ can be regarded as a point in $E^n$ under the closed immersion (\ref{coordi}). For a point $x=(x_v)\in X(F_S)=\prod_{v\in S}X(F_v)$, we define its height to be \[
H_S(x)\coloneqq\max_{v\in S}H_v(x_v),
\] where $H_v(x_v)=\max_{1\leqslant i\leqslant n}\{|z_i^v|_v\}$ if $x_v$ corresponds to $(z_1^v,\ldots,z_n^v)\in F_v^n$ under the chosen embedding.

In particular, for a rational point $x\in X(F)$ with coordinate $(z_1, \ldots, z_n)\in F^n$, one has
\[
H_S(x)= \max_{v\in S, 1\leqslant i\leqslant n} \{ \ |z_i|_v  \} .
\] 

Our main interest will be in the asymptotic behavior of the number
\[
N({\calX}, q^n)= \# \{x\in{\calX} (\calO_S): \  H_S(x)\leqslant q^n\}
\] of $S$-integral points on $\calX$ with bounded height, as $n\to \infty$. As in the case of number fields, under suitable assumptions we will obtain an asymptotic formula given by the sum of products of local densities twisted by Brauer elements, which we now describe in more detail.


Write the Brauer--Manin pairing (see \cite[\S~5.2]{Skorobogatov01}) in a multiplicative way as
$$ \aligned   X (\bfA_F) \times \Br(X) & \longrightarrow \mu_\infty = \varinjlim_{n} \mu_n \subset \mathbb{C}^\times, \\
((x_v)_{v\in \Omega_F}, \alpha ) & \mapsto \prod_{v\in \Omega_F} \alpha (x_v) \endaligned
$$
where the $\alpha(x_v)$'s are all roots of unity and $\alpha(x_v)=1$ for almost all $v\in \Omega_F$ via the natural isomorphism $$\mathbb Q/ \mathbb Z \xrightarrow{\simeq} \mu_\infty . $$ Then one can view any element in $\Br(X)$ as a locally constant $\mathbb{C}$-valued function on $X(\bfA_F)$. For any subset $B$ of $\Br(X)$, we write
$$ X(\bfA_F)^{B} = \{ (x_v)\in X(\bfA_F): \ ((x_v),\alpha)=1 \ \text{for all $\alpha\in B$} \} . $$

\begin{defn} For any $\xi\in \Br(X)$, define
\[
I_v ({\calX},  \xi)= \int_{{\calX}(\calO_v)} \xi
\d m^X_v
\] for any $v\in \Omega_F\backslash S$ and
$$ I_{S}(X, q^n,\xi) = \int_{X(F_S, q^n)} \xi \d m^X_{S} $$ for $n>0$
(see Definition \ref{tam} for the precise description of the measures $ m^X_v$ and $m^X_S$). Here $X(F_S , q^n)\coloneqq \{ x\in X(F_S):  H_S(x)\leqslant q^n \}$ is easily seen to be a compact subset of $X(F_S)$ with product topology.
\end{defn}

One final ingredient is needed before we can state the main result of this paper. For a connected reductive group $G$ over $F$, denote by $\widehat{G}(\overline{F})$ and $\widehat{G}(F)$ the groups of $\overline{F}$-characters and $F$-characters of $G$ respectively, and by $t_G$ the rank of $\widehat{G}(F)$. If $\widehat{G}(\overline{F})\neq0$, let $\rho_G$ denote the representation of $\mathrm{Gal}(\overline{F}/F)$ in the space $\widehat{G}(\overline{F})\otimes\mathbb{C}$. For any place $v\in\Omega_F$, let $L_v(\rho_G, s)$ be the Euler factor at $v$ of the Artin $L$-function associated with $\rho_G$. One then has $L_v(\rho_G,1)\neq0$ by \cite[I.4.8(b)]{Oester84}. In addition, the product $L_S(\rho_G, s)=\prod\limits_{v\in \Omega_F\setminus S}L_v(\rho_G, s)$, for $S$ a finite subset of $\Omega_F$, admits a meromorphic continuation to the whole complex plane with a pole at 
$s=1$ of order $t_G$ (\cite[I.4.5]{Oester84}). The following factors are used in the definition of the Tamagawa measure (see \S~\ref{sec-measure}):
\begin{equation}\label{ri}
 \lambda_v^G \coloneqq L_v(\rho_G, 1)^{-1}, \ \ \  r_G \coloneqq \lim_{s\rightarrow 1}(s-1)^{t_G}L_S(\rho_G,s).
\end{equation} 

\begin{remark}\label{rmk-inner form}
(1). If $\widehat{G}(\overline{F})=0$ (namely $G$ is semi-simple), there is no need of convergence factors in the definition of the Tamagawa measure on $G(\mathbf{A}_F)$ (cf. \cite[Prop.~1.2.5]{Gaitsgory-Lurie19}), and we use the convention that $\lambda_v^G=1$ and $r_G=1$ in this case.\\
(2). If $G'$ is an inner form of $G$, one has $\lambda_v^G=\lambda_v^{G'}$ and $r_G=r_{G'}$. Indeed, there is by assumption an $\overline{F}$-isomorphism $f\colon \overline{G} \to \overline{G'}$ such that $f^{-1}\circ \sigma f=\mathrm{inn}(g_\sigma)$ is inner for every $\sigma\in \Gal(\overline{F}/F)$. Therefore, on the level of points we have $(\sigma f)(x)=f(g_\sigma)f(x)f(g_\sigma)^{-1}$, and hence $\chi((\sigma f)(x))=\chi(f(x))$ for any $\chi\in \widehat{G'}(\overline{F})$. Thus $\sigma \chi\circ f=\sigma\chi \circ \sigma f=\sigma(\chi\circ f)$ for all $\sigma$ and $\chi$, in other words the induced isomorphism $f^*\colon \widehat{G'}(\overline{F})\to \widehat{G}(\overline{F})$ is Galois equivariant, and hence descends to an isomorphism $\widehat{G'}(F)\to \widehat{G}(F)$. In particular, $\rho_G$ and $\rho_{G'}$ are isomorphic as Galois representations, and $r_G=r_{G'}$.

\end{remark}

\subsection*{Main result}
The main result of this paper is the following asymptotic formula.
\begin{thm}\label{th1}
Let $G$ be a simply connected and almost $F$-simple linear algebraic group over $F$ such that $G(F_S)$ is not compact. Let $H$ be a subgroup of fixed points of some involution of $G$. Set $X=H\backslash G$ and let $\calX$ be a finite-type separated scheme over $\calO_S$ whose generic fiber is $X$. Suppose that $H$ is connected and has no non-trivial $F$-characters. Then we have
\[
N({\calX}, q^n) \sim r_H\cdot q^{(1-\eta_F)\dim X}\sum_{\xi\in \Br_{1,P}(X,G)} (\prod_{v \notin S}  I_v(\calX{}, \xi)) \cdot I_{S} (X, q^n, \xi)
\]
as $n\rightarrow \infty$.
\end{thm}

\begin{remark}
    Compared to the number field case (see \cite[Theorem~1.5]{Wei-Xu16}), a notable difference is that only the elements in $\Br_{1,P}(X,G)$, instead of the whole Brauer group $\Br(X)$, contribute to our formula. The reason for this phenomenon is that, for affine varieties over function fields, the Brauer group is too large so that any adelic point orthogonal to the whole Brauer group is already a rational point (\cite[Proposition~4.3]{Harari-Voloch13}). As a consequence, one has to use the group $\Br_{1,P}(X,G)$ instead of $\Br(X)$ in Proposition~\ref{equiv} which is crucial for the proof of Theorem~\ref{th1}.
\end{remark}

As application of Theorem~\ref{th1}, we study the asymptotic formula in the particular case where $F=\bF_q(t)$, $S=\{\infty\}$, and $\calX$ is defined by the equation $x_1^2+\cdots+x_r^2=a$, with $a\in \calO_S$ and $r\geqslant3$. When $r\geqslant4$, the group $\Br_{1,P}(X,G)$ is trivial and the asymptotic formula reduces to the computations of local densities (see Proposition~\ref{localdensity}). Whereas when $r=3$, one has $\Br_{1,P}(X,G)\simeq \mathbb{Z}/2$ and the situation is more subtle: if $a$ admits a prime factor of odd power in its factorization, it turns out that still the asymptotic formula reduces to the product of local densities (see Proposition~\ref{prop-oddlocint}); otherwise we show by explicit example that the summand corresponding to the nontrivial Brauer element can make nontrivial contribution to the counting (see Example~\ref{ex-even}).

The organization of the paper is as follows. In \S~\ref{sec-measure}, we recall some necessary facts about Tamagawa measure on homogeneous spaces. Subsequently, we summarize basic properties of quadratic forms and lattices in \S~\ref{subsec-qf} needed in the last section. In \S~\ref{subsec-orb}, we follow the approach of \cite{Wei-Xu16} to obtain a natural action on $\calX(\calO_S)$ by a suitable arithmetic group $\Gamma$, and describe the collection of orbits. Then in \S~\ref{subsec-counting}, we invoke the equidistribution result \cite{Benoist-Oh12}, which helps to count inside a single $\Gamma$-orbit, thus establishing the asymptotic formula of $N({\calX}, q^n)$ by using Brauer–Manin obstruction. In the final \S~\ref{sec-example}, as mentioned above, the asymptotic behavior of the representation number of a sum of squares is studied as a concrete example.

\subsection*{Acknowledgements}
The authors would like to thank Prof. Fei Xu and Prof. Xiaowen Hu for helpful 
discussions and generous support. They are also grateful to an anonymous referee for 
valuable suggestions to improve the presentation of this paper. The second author was 
supported by NSFC Grant (No. 12371063).

\section{Preliminaries}
\subsection{Homogeneous space and Tamagawa measure}\label{sec-measure}

For later use, we collect some facts about Tamagawa measure on homogeneous spaces.
Basic references are \cite[Section~1]{Borovoi-Rudnick95}, \cite[Chapitre~I]{Oester84} and \cite[Chapter~II]{Weil82}.

\begin{defn}\label{def:gauge_form}
For a geometrically integral smooth algebraic variety $V$ over $F$,
a gauge form on $V$ is an algebraic differential form of degree $\dim(V)$, everywhere regular and nonzero.
\end{defn}
Suppose given a connected reductive group $G$ and a connected reductive subgroup $H$ of $G$. Then $G$ and $H$ admit invariant gauge forms. Moreover, the homogeneous space $X=H\backslash G$ admits a $G$-invariant gauge form (cf. \cite[\S 2.4]{Weil82}).

\begin{defn}[{\cite[\S~2.4, p. 24]{Weil82}}]\label{deg:algmatch}Invariant gauge forms $\omega_G$ on $G$, $\omega_H$ on $H$, and $G$-invariant gauge form $\omega_X$ on $X$ are said to match together algebraically if $\omega_G=\omega_X\cdot \omega_H$.
\end{defn}
For any choice of invariant gauge forms $\omega_G$ on $G$ and $\omega_H$ on $H$, we may multiply a given $\omega_X$ by a suitable constant in $F^\times$ so that $\omega_G = \omega_X \cdot \omega_H$.

Given a gauge form  $\omega_V$ on a variety $V$ over $F$, one can define a measure $|\omega_V|_v$ on $V(F_v)$ for each $v \in \Omega_F$ (see \cite[\S~2.2.1]{Weil82}). The following lemma will be used without explict mention in our proof of the main theorem. For reference see the discussion in \cite[the first paragraph of p. 47]{Borovoi-Rudnick95} or \cite[Proposition~2.5]{Lee25}.
\begin{lemma}
Suppose the invariant gauge forms $\omega_G$ on $G$, $\omega_H$ on $H$ and $\omega_X$ on $X$ match together algebraically. Then for any $v\in\Omega_F$ we have
    \begin{equation}\label{eq:topologicallymatch}
    \int_{G(F_v)}f(g)\ \d |\omega_G|_v=\int_{H(F_v)\backslash G(F_v)} \d |\omega_X|_v\int_{H(F_v)}f(hg) \ \d |\omega_H|_v.
    \end{equation}
    For a triple of measures $(|\omega_G|_v,|\omega_H|_v,|\omega_X|_v)$ satisfying (\ref{eq:topologicallymatch}), we say that $|\omega_G|_v,|\omega_H|_v,$ and $|\omega_X|_v$ match together topologically.
\end{lemma}

\begin{defn}\label{tam}
(1). Let $G$ be a connected reductive group over $F$ and  $\omega_{G}$ be an invariant gauge form on $G$.
The Tamagawa measure $m^G$ on $G(\bfA_F)$ is defined as
\[
m^G=m_S^G\cdot r_G^{-1}q^{(1-\eta_F)\dim G}\prod_{v\in \Omega_F\setminus S}m^G_v,
\]where\[
\begin{cases}
    m^G_v=(\lambda_v^G)^{-1}|\omega_G|_v;\\
     m_S^G=\prod_{v\in S}|\omega_G|_v;
\end{cases}
\]and the factors $r_G$ and $\lambda_v^G$ are as defined in \eqref{ri}. 
The Tamagawa measure $m^G$ is independent of choices of $\omega_G$ and $S$
(cf. \cite[I.4.7]{Oester84}).  
For an id\`ele $a \in \bfA^{\times}_F$ its norm is defined by $|a|=\prod_{v\in \Omega_F}|a_v|_v$. Let $G(\bfA_F)^1$ denote the set of all $b\in G(\bfA_F)$ such that $|\chi(b)|=1$ for any $F$-character $\chi\colon G \to \mathbb{G}_m$. The Tamagawa number of $G$ is then defined by $\tau(G)=m^G(G(F)\backslash G(\bfA_F)^1)$.\\
(2). Let $\omega_X$ be a $G$-invariant gauge form on $X\simeq H\backslash G$.
The Tamagawa measure $m^X$ on $X({\bf A}_F)$ is defined as
\begin{align}\label{Tam-measure}
    m^X=m_S^X\cdot r_X^{-1}q^{(1-\eta_F)\dim X}\prod_{v\in\Omega_F\backslash S} m_{v}^X,
\end{align}
where similarly
\[
\left\{
\begin{array}{l}
     m_{v}^X=(\lambda_v^X)^{-1}|\omega_X|_v; \\
     m_S^X=\prod_{v\in S}|\omega_X|_v;\\
\end{array}
\right.
\]
and $r_X=\dfrac{r_G}{r_H}$, $\lambda_v^X=\dfrac{\lambda_v^G}{\lambda_v^H}$.
\end{defn}
From now on, we always fix an invariant gauge form $\omega_G$ on $G$, an invariant gauge form $\omega_H$ on $H$, and a $G$-invariant gauge form $\omega_X$ on $X$ so that $\omega_G=\omega_X\cdot \omega_H$.
Then the local measures $m^G_v$ on $G(F_v)$, $m^H_v$ on $H(F_v)$, and $m^X_v$ on $X(F_v)$ match together topologically.

\subsection{Quadratic forms and lattices}\label{subsec-qf}

In this part we gather some basic notions and facts about quadratic forms, mostly specialized to the case of a global function field as base field, in preparation for \S~\ref{sec-example}. Terminology is standard if not explained, and is generally adopted from \cite{Lam05} or \cite{OMeara00}.

\noindent\textbf{Quadratic spaces.} A quadratic space over a field $k$ of $\car(k)\neq2$ is a pair $(V,Q)$ consisting of a finite dimensional $k$-vector space $V$ and a quadratic map $Q\colon V\to k$. Its associated symmetric bilinear form $B\colon V\times V\to k$ is defined by $B(v,w)\coloneqq \tfrac{1}{2}(Q(v+w)-Q(v)-Q(w))$. When a basis of $V$ is fixed, the value of $Q$ at a vector is given by a homogeneous polynomial of degree $2$, i.e. a quadratic form, in the coordinates of this vector. In this sense we shall use the terms quadratic form and quadratic space interchangeably.

The discriminant $dV$ of a quadratic space $(V,Q)$ over $k$ (or $d(f)$ for a quadratic form $f$) is defined to be the class of $\det((B(e_i,e_j)))$ in $k^\times/k^{\times2}\cup\{0\}$, where $(e_i)$ is any basis of $V$ over $k$. When $dV\neq0$, the quadratic space is called non-degenerate. All the quadratic spaces are assumed to be non-degenerate in this paper unless otherwise stated.

A nonzero vector $v$ in a quadratic space $(V,Q)$ with $Q(v)=0$ is called an isotropic vector. This corresponds to a non-trivial zero of the quadratic form. A quadratic space is called isotropic if it admits an isotropic vector, and anisotropic otherwise. For $F$ a global function field with set of places $\Omega_F$, isotropy of quadratic forms over $F$ is equivalent, by the famous Hasse-Minkowski theorem \cite[VI.3.1]{Lam05}, to isotropy over the local fields $F_v$ for all $v\in\Omega_F$. While over every local function field $F_v$, with uniformizer $p_v$, any quadratic form is up to isometry of the form $f(x_1,\ldots,x_n)=\sum_{i=1}^m a_ix_i^2+p_v\sum_{j=m+1}^{n} b_jx_j^2$ for some $v$-adic units $a_i$ and $b_j$. And moreover, isotropy of $f$ is equivalent to isotropy of one of the forms $\sum_{i=1}^m \overline{a_i}x_i^2$ and $\sum_{j=m+1}^{n} \overline{b_j}x_j^2$ over the residue field of $F_v$ (cf. \cite[Proposition~VI.1.9]{Lam05}). Since any quadratic form of rank at least $3$ is isotropic over a finite field \cite[Proposition~II.3.4]{Lam05}, we have thus justified the following

\begin{lemma}\label{lem-soq}
    The sum of $r\geqslant3$ squares $x_1^2+\cdots+x_r^2$ is isotropic over $F$.
\end{lemma}

\noindent\textbf{Quadratic lattices.} We next recall some elements of the theory of integral quadratic forms. The theory can be developed over a general Dedekind domain $R$, but we focus on the case where $R$ is a PID with $2\in R^{\times}$, because the only application in this paper will be to the polynomial ring $R=\mathbb{F}_q[t]$ (and its completions). A quadratic lattice over $R$ is defined to be a finitely generated free $R$-module $M$ together with a quadratic map $Q$ on the vector space $V\coloneqq M\otimes_{R}\mathrm{Frac}(R)$. For a quadratic lattice $M$ over $R$ with $\mathrm{Frac}(R)=F$, we will denote the underlying quadratic spaces $M\otimes_{R}F$ (resp. $M\otimes_{R}F_v$) by $FM$ (resp. $F_vM$).

A quadratic lattice $(M,Q)$ with $Q(M)\subseteq R$ is called unimodular if the matrix $(B(e_i,e_j))$ of $Q$ with respect to any $R$-basis $e_1,\ldots,e_n$ of $M$ is unimodular, namely, $\det((B(e_i,e_j)))\in R^{\times}$. The discriminant $dM$ of a unimodular lattice $M$ is the class of $\det((B(e_i,e_j)))$ in $R^{\times}/R^{\times2}$. For a sublattice $J$ of a quadratic lattice $M$, we say $J$ splits $M$ if there exists a sublattice $J'$ of $M$ such that $M= J\perp J'$, namely, $M=J\oplus J'$ and $B(J,J')=0$. Clearly if $J$ splits a unimodular lattice $M$ as $M=J\perp J'$, then $dM=dJ\cdot dJ'$.

One of the main reasons that make the notion of a (uni)modular lattice useful is that it often allows to split off sublattices as orthogonal summands. 

\begin{prop}[{cf. \cite[82:15]{OMeara00}}]\label{prop-split}
    Let $J$ be a unimodular sublattice of a quadratic lattice $M$, then $J$ splits $M$ if and only if $B(J,M)\subseteq R$. In particular, this is automatic if $Q(M)\subseteq R$, e.g. if $M$ is unimodular.
\end{prop}

\noindent\textbf{Scheme of representations.} Given two quadratic spaces $(W,q)$ and $(V,Q)$ over a field $k$, a representation (or isometry) from $W$ to $V$ is a $k$-linear map $\sigma\colon W\to V$ with $q=Q\circ \sigma$. A quadratic space $(W,q)$ is said to be represented by $(V,Q)$ if such a representation exists. The isometries from $(V,Q)$ to itself form the orthogonal group $\rO(V,Q)=\rO(V)$ of $V$, and the special orthogonal group is $\rO^+(V)\coloneqq \Ker[\det\colon\rO(V)\to k^\times]$. 
Similarly, for quadratic lattices $(N,q)$ and $(M,Q)$ over $R$, a representation is an $R$-linear map $\sigma\colon N\to M$ with $q=Q\circ \sigma$. One can define $\rO(M)$ and $\rO^+(M)$ for a lattice $M$ in the same way as for quadratic spaces.

Suppose $(N=Rw_1\oplus\cdots\oplus Rw_m,q)$ and $(M=Rv_1\oplus\cdots\oplus Rv_n,Q)$ are quadratic lattices over $R$. Let $f$ and $g$ be the corresponding quadratic forms of $Q$ and $q$ with respect to these given bases, namely $f(x_1,\ldots,x_n)=Q(x_1v_1+\cdots+x_nv_n)$ and $g(y_1,\ldots,y_m)=q(y_1w_1+\cdots+y_mw_m)$. Then a representation $\sigma\colon N\to M$ is uniquely determined by linear forms $x_i=l_i(y_1,\ldots,y_m)$ for $i=1,\ldots,n$, with coefficients in $R$, such that \begin{equation}\label{eq-repscheme}
f(l_1(y_1,\ldots,y_m),\ldots,l_n(y_1,\ldots,y_m))=g(y_1,\ldots,y_m).
\end{equation} In particular, when $m=1$, so that $g=ay^2$, such a representation is equivalent to an $n$-tuple $(x_1,\ldots,x_n)\in R^n$ with $f(x_1,\ldots,x_n)=a$. On the other hand, the identity \eqref{eq-repscheme} leads to equations in the coefficients of the linear forms $l_i$, which define an $R$-scheme $\calX$. We have thus identified the set $\calX(R)$ of $R$-points of $\calX$ with that of representations of $(N,q)$ by $(M,Q)$. More generally, for any $R$-algebra $S$ which is also a PID, the set $\calX(S)$ is identified with representations of $N\otimes_{R}S$ by $M\otimes_{R}S$ as quadratic $S$-lattices. Similarly, the generic fiber $X$ of $\calX$ can be viewed as the variety of representations of $(N\otimes_{R}\mathrm{Frac}(R),q)$ by $(M\otimes_{R}\mathrm{Frac}(R),Q)$.

One can consult \cite{CT-Xu09} for more details on the above interpretation, where it was used by the authors to show that many counter-examples to the local-global principle for integral representations of quadratic forms can be accounted for via the Brauer-Manin obstruction to strong approximation for $\calX$.

\section{Counting integral points}\label{sec-counting}

In this section we follow the notations and conventions settled in \S~\ref{sec-introduction}.

\subsection{Orbits}\label{subsec-orb}
Fix a finite set $S_0$ of places containing $S$ such that there are group schemes $\calG$ and $\calH$ of finite type over $\calO_{S_0}$ with generic fibers $G$ and $H$ respectively and satisfying \[\calX_{\calO_{S_0}}\simeq\calH\backslash\calG\] and $P\in\calX(\calO_v)$ for all $v\notin S_0$. For each $v\in S_0\backslash S$, fix a group scheme $\calG_v$ of finite type over $\calO_v$ such that $\calG_v\times_{\calO_v}F_v\simeq G\times_{F}F_v$.

\begin{defn}
    For each $v\in\Omega_F$, define
    \[
    \St(\calX)_v\coloneqq
    \begin{cases}
        \calG(\calO_v), &v\notin S_0\\
        \{g\in\calG_v(\calO_v): \calX(\calO_v)=\calX(\calO_v)\cdot g\}, &v\in S_0\backslash S\\
        G(F_v). &v\in S
    \end{cases}
    \]
\end{defn}

Then $\St(\calX)_v$ is an open subgroup of $G(F_v)$ for each $v\in\Omega_F$ and it is compact for $v\notin S$. By definition, $\calX(F_v)$, resp. $\calX(\calO_v)$, is stabilized by $\St(\calX)_v$ under the action of $G$ on $X$ for $v\in S$, resp. $v\notin S$.

\begin{lemma}\label{localorb}
    If $H$ is reductive, then\\
    (1). $\#[\calX(F_v)/\St(\calX)_v]$ is finite for any $v\in S$.\\
    (2). $\#[\calX(\calO_v)/\St(\calX)_v]$ is finite for any $v\in\Omega_F\setminus S$.
\end{lemma}

\begin{proof}
    When $v\in S$, one has $\calX(F_v)=X(F_v)$ and $\St(\calX)_v=G(F_v)$, hence $\calX(F_v)/\St(\calX)_v$ corresponds bijectively to $\Ker[\rH^1(F_v,H)\to\rH^1(F_v,G)]$. Therefore \[\#[\calX(F_v)/\St(\calX)_v] \leqslant \#\rH^1(F_v,H),\] and the latter is finite by \cite[Prop.7.1.3]{Conrad12} and \cite[Prop.A.2.11]{Conrad-Gabber-Prasad15}.

    For any $x\in {\calX}(\calO_{v})$ with $v\notin S$, the morphism induced by the point $x$
\[f_x\colon   G\simeq \{x \} \times_F G \longrightarrow  X \times_F G\xrightarrow{m} X\]  is dominant and smooth.
By \cite[Proposition~3.5.73(ii)]{Poonen18}, we have that $f_x$ is an open map over $F_v$ points. This implies that $x\St({\calX})_v$ is open in ${\calX}(\calO_{v})$. By compactness of ${\calX}(\calO_{v})$, one concludes that $\#[{\calX}(\calO_{v})/ \St({\calX})_v]$ is finite.
\end{proof}

Define \[\St({\calX})\coloneqq \prod_{v\in \Omega_F} \St({\calX})_v.\]
Then $\St({\calX})$ is an open subgroup of $G({\bf A}_F)$. We set
\[
{\calX} \cdot {\sigma_{\bf A}} \coloneqq \prod_{v\in S}\big({\calX}(F_{v}) \cdot  \sigma_{v}\big)\times \prod_{v\in \Omega_F\setminus S}\big({\calX}(\calO_{v}) \cdot  \sigma_{v}\big)\subseteq
 X ({\bf A}_F)
\] for any $\sigma_{\bf A}=(\sigma_v)_{v\in \Omega_F}\in G({\bf A}_F)$. In particular,
\[
{\calX} \cdot {1_{\bf A}} =\prod_{v\in S}{\calX}(F_{v})\times\prod_{v\in \Omega_F\backslash S} {\calX}(\calO_{v}).
\]
It is clear that $\sigma_{\bf A}^{-1}\St({\calX})\sigma_{\bf A}$ acts on $ {\calX} \cdot \sigma_{\bf A}$ (on the right).

 \begin{coro} \label{adelorb} The number of orbits \[
 \#[{\calX} \cdot \sigma_{\bf A}/\sigma_{\bf A}^{-1}\St({\calX})\sigma_{\bf A}]
 \] is finite.

\end{coro}

\begin{proof} 
Write $\sigma_{\bf A}=(\sigma_v)_{v\in \Omega_F}$. Then one has the componentwise bijection
\[
{\calX} \cdot \sigma_{\bf A}/ \sigma_{\bf A}^{-1}  \St({\calX}) \sigma_{\bf A}
 \simeq [\prod_{v\in S} {\calX}(F_{v}) \cdot \sigma_v / \sigma_v^{-1} \St({\calX})_v \sigma_v ]\times [\prod_{v\not\in S} {\calX}(\calO_{v}) \cdot \sigma_v /\sigma_v^{-1} \St({\calX})_v \sigma_v ].  
\]
There is a finite subset $S_1\supseteq S_0$ such that for all $v\not \in S_1$ one has $\sigma_v\in {\calG}(\calO_{v})$ and, in particular, ${\calX}(\calO_{v}) \cdot \sigma_v /\sigma_v^{-1} \St({\calX})_v \sigma_v={\calX}(\calO_{v})/ {\calG}(\calO_{v})$. Moreover, enlarging $S_1$ if necessary, one may assume that $\calH$ is smooth outside $S_1$ (see \cite[Theorem~3.2.1(ii)]{Poonen18}).

As a connsequence, for $v\not \in S_1$ one has $\rHet^1(\calO_{v}, {\calH})=1$ by Hensel's lemma \cite[Theorem~3.5.63(a)]{Poonen18} and Lang's theorem \cite[Theorem~5.12.19(a)]{Poonen18}, therefore the canonical exact sequence of pointed sets
\[
1\rightarrow {\calH}(\calO_{v}) \rightarrow {\calG} (\calO_{v}) \rightarrow {\calX}(\calO_{v}) \rightarrow \rHet^1(\calO_{v}, {\calH})
\]
gives $\#[{\calX}(\calO_{v}) \cdot \sigma_v /\sigma_v^{-1} \St({\calX})_v \sigma_v]=\#[{\calX}(\calO_{v})/ {\calG}(\calO_{v})]=1$.
Now the result follows from Lemma \ref{localorb}, as the set of orbits ${\calX}(F_{v}) \cdot \sigma_v / \sigma_v^{-1} \St({\calX})_v \sigma_v$ for $v\in S$, resp. ${\calX}(\calO_{v}) \cdot \sigma_v /\sigma_v^{-1} \St({\calX})_v \sigma_v$ for $v\in S_1\backslash S$, is easily seen to be in bijection with ${\calX}(F_{v})/\St({\calX})_v$, resp. ${\calX}(\calO_{v})/\St({\calX})_v$.
\end{proof}

\begin{defn} 
Put $\Gamma \coloneqq G(F) \cap \St({\calX})$, the intersection being taken inside $G({\bfA}_F)$. 
\end{defn}

 It is clear that $\Gamma$ acts on ${\calX}(\calO_S)$.

\begin{defn} \label{locquiv} 
For any $x,y \in {\calX}(\calO_S)$, we define the equivalence relation:
\[
x\sim y \ \ \ \Leftrightarrow \ \ \ x=y \cdot s_{\bf A} \text{ for some }s_{\bf A}\in \St({\calX}).
\]   The set of equivalence classes is denoted by ${\calX}(\calO_S)/\sim $.

\end{defn}


\begin{prop} \label{equiv}
    If $G$ is semi-simple and simply connected such that $G'(F_S)$ is not compact for any simple factor $G'$ of $G$, then the diagonal map
\begin{equation}\label{diag}
    {\calX}(\calO_S)/\sim  \ \ \longrightarrow  \ ({\calX} \cdot 1_{\bf A})^{\Br_{1,P}(X, G)}/\St({\calX})
\end{equation}
is bijective. In particular, the number of orbits $\#[\calX(\calO_S)/\sim]$ is finite.

\end{prop}

\begin{proof}
By \cite[Proposition~3.3, Definition~3.6, Theorem~3.7]{Demarche-Harari22}, the action of a $g_{\bfA}\in G({\bfA}_F)$ on any $x_{\bfA}\in X({\bfA}_F)$ does not affect (up to a universal sign) the Brauer-Manin pairing with a fixed Brauer element in $\Br_{1,P}(X,G)$, namely \[
((x_v\cdot g_v)_v,\alpha)=\pm((x_v)_v,\alpha) \text{ for any }\alpha\in\Br_{1,P}(X,G).
\] In particular, one sees that the subset $({\calX} \cdot 1_{\bf A})^{\Br_{1,P}(X, G)}$ is stabilized under the action of $\St(\calX)$ on $\calX\cdot 1_{\bfA}$. By definition it is clear that the map \eqref{diag} is injective.

To see that this map is also surjective, let $O$ be an $\St({\calX})$-orbit in $({\calX} \cdot 1_{\bf A})^{\Br_{1,P}(X, G)}$ and fix $(x_v) \in O$. Since $O$ is open in ${\calX} \cdot 1_{\bf A}$ by the proofs of Lemma~\ref{localorb} and Corollary~\ref{adelorb}, it suffices to find an integral point very close to $(x_v)$. By \cite[Theorem~5.8 (1)]{Demarche-Harari22} and the strong approximation property of $G$ outside $S$ (see \cite[Theorem~A]{Prasad77}), there exist $(g_v)\in G({\bfA}_F)$ and $N\in X(F)$ such that $N\cdot g_v=x_v$ for any $v\in\Omega_F$. Now a similar argument as in the proof of \cite[Theorem~3.7 (b)]{CT-Xu09} shows that there is an integral point $M\in {\calX}(\calO_S)\cap O$, and the surjectivity of \eqref{diag} follows.

In particular, one has \[
\#[\calX(\calO_S)/\sim]=\#[({\calX} \cdot 1_{\bf A})^{\Br_{1,P}(X, G)}/\St({\calX})]\leqslant \#[{\calX} \cdot 1_{\bf A}/\St({\calX})]<\infty
\] by Corollary~\ref{adelorb}.
\end{proof}

Write the equivalence class decomposition with respect to $\sim$ as \begin{equation}\label{partition}
{\calX}(\calO_S) \  =  \ \coprod_{i}  ({\calX}(\calO_S)\cap x_i \St({\calX}))  
\end{equation}  with $x_i\in {\calX}(\calO_S)$.  Within each class ${\calX}(\calO_S)\cap x_i \St({\calX})$ we can further define an equivalence relation $\sim_G$ by
\[y\sim_G z \ \ \Leftrightarrow \ \ y=z\cdot g \text{\ \ for some\ \ }g\in G(F).\]

\begin{prop} \label{sha}
One has the following inequality
\[
\# [({\calX}(\calO_S)\cap x_i \St({\calX}))/\sim_G] \leqslant  \#\Sha^1(F,H)<\infty.
\]

If moreover $G$ is semi-simple simply connected and $G'(F_S)$ is not compact for any simple factor $G'$ of $G$, then equality holds.
\end{prop}

\begin{proof} 
The proof of \cite[Proposition~2.11]{Wei-Xu16} goes through just as fine in the function field case, thanks to the fact that $\Sha^1(F,H)$ is finite by \cite[Theorem~1.3.3(i)]{Conrad12}.
\end{proof}

One can further decompose \[
{\calX}(\calO_S)\cap x_i \St({\calX})  = \coprod_{j} (y_{ij} G(F) \cap x_i \St({\calX}))= \coprod_{j} (y_{ij} G(F) \cap y_{ij} \St({\calX}))
\]
with $y_{ij} \in {\calX}(\calO_S)$.

\begin{prop}\label{refine}  With the above notation, there is a bijection
\[
(y_{ij} G(F) \cap y_{ij} \St({\calX}))/\Gamma   \ \xlongrightarrow{\simeq}  \  H_{ij}(F) \backslash (H_{ij}({\bf A}_F)\cap G(F)\St({\calX}))/ (H_{ij}({\bf A}_F)\cap \St({\calX}))
\] where $H_{ij}$ is the stabilizer of $y_{ij}$ in $G$. Also, the sets on both sides are finite.


\end{prop}

\begin{proof} 
The proof of \cite[Proposition 2.12]{Wei-Xu16} is still valid over global function fields. For the finiteness assertion, note that \[
(H_{ij}({\bf A}_F)\cap \St({\calX}))\supseteq H_{ij}(F_S)\cdot (H_{ij}({\bfA}_{F}^S)\cap\prod_{v\notin S}\St(\calX)_v)
\] and $H_{ij}({\bfA}_{F}^S)\cap\prod_{v\notin S}\St(\calX)_v$ is a compact open subgroup of $H_{ij}({\bfA}_{F}^S)$ by definition of $\St(\calX)$. It therefore follows from \cite[Theorem~1.3.1]{Conrad12} that $\#[(y_{ij} G(F) \cap y_{ij} \St({\calX}))/\Gamma]<\infty$.
\end{proof}

\begin{remark}
    Combining Proposition~\ref{equiv}, Proposition~\ref{sha} and Proposition~\ref{refine}, one obtains $\#[\calX(\calO_S)/\Gamma]<\infty$.
\end{remark}

\subsection{The asymptotic formula of \texorpdfstring{$N({\calX}, q^n)$}{N(X, qⁿ)}}\label{subsec-counting}
\begin{lemma}\label{3.1}
For $X=H\backslash G$ with fixed rational point $P\in X(F)$, suppose $G$ is semi-simple and simply connected and $H$ is connected reductive. Then $\Br_{1,P}(X, G)$ is finite and \[
\Br_{1,P}(X, G)\simeq\Pic(H).
\] If moreover $H$ is semi-simple simply connected, then\[
    \Pic(H)=\Br_{1,P}(X, G)=0.
    \]
\end{lemma}
\begin{proof}
Since $G$ is semi-simple and simply connected, we have $\Pic(G)=0$ and $\Br_1(G)=\Br(F)$ by \cite[Lemma~6.9 (iv)]{Sansuc81}. The exact sequence (6.10.1) in [\textit{loc. cit.}, Proposition~6.10] gives \[
0\to \Pic(H)\to \Br(X)\xrightarrow{\pi^*} \Br(G) 
\] which in particular shows that $\Pic(H)$ embedds into $\Br_1(X,G)$. But $\pi^*$ maps $\Br_1(X,G)$ into $\Br_1(G)=\Br(F)$, hence the above sequence induces \[
\Pic(H)\xrightarrow{\sim}\Ker[\pi^*: \Br_{1}(X, G) \to  \Br(F)]=\Br_{1,P}(X,G)
\]Since $\Pic(H)$ is finite by \cite[Theorem~1.3 and Remark~1.4]{Rosengarten21b}, we get that
$\Br_{1,P}(X, G)$ is finite. 

If $H$ is semi-simple simply connected, we have $\Pic(H)=0$ again by \cite[Lemma~6.9 (iv)]{Sansuc81}, and hence $\Br_{1,P}(X,G)=0$.
\end{proof}

\begin{lemma}[Compare {\cite[Lemma 4.1]{Wei-Xu16}}]\label{3.2}
Suppose $X=H\backslash G$ where both $G$ and $H$ are connected reductive groups without non-trivial characters over $F$. For any finite subgroup $B$ of $\Br(X)/\Br(F)$, one has
\[
\frac{r_H}{r_G}q^{(1-\eta_F)\dim X}\sum_{\xi \in B} (\prod_{v\notin S} \int_{{\calX}(\calO_v)} \xi \d m^X_v \cdot \int_{X(F_S, q^n)} \xi \d m^X_{S}) =  \# B  \int_{(\prod_{v\notin S}{\calX} (\calO_{v}) \times X(F_S, q^n))^B} \d m^{X}
\]
where
\[
(\prod_{v\notin S}{\calX} (\calO_{v}) \times X(F_S, q^n))^B= (\prod_{v\notin S} {\calX}(\calO_{v}) \times X(F_S, q^n))\cap ({\calX} \cdot 1_{\bfA} )^B.
\]
\end{lemma}

\begin{proof}
For the same reason as in the proof of \cite[Lemma 4.1]{Wei-Xu16}, one has \[
\# B  \int_{(\prod_{v\notin S}{\calX} (\calO_{v}) \times X(F_S, q^n))^B} \d m^{X}= \sum_{\xi\in B}\int_{\prod_{v\notin S}{\calX} (\calO_{v}) \times X(F_S, q^n)} \xi\d m^{X},
\]
the result then follows from the definition of $m^X$ in \eqref{Tam-measure}.
\end{proof}

We assume the following equi-distribution property for arbitrary $x\in \calX(\calO_S)$:

\begin{equation}\label{equi-dist}
\# \{y \in  x \cdot \Gamma: H_S(y)\leqslant q^n \} \sim  \frac{m^{H_x}_S( \Gamma_{H_x}\backslash H_x(F_S)) }{m^G_{S} (\Gamma \backslash G(F_S)) } m^X_{S} (x \cdot G(F_{S}) \cap X(F_S,  q^n))
\end{equation}
as $n \rightarrow \infty$, where $H_x$ is the stabilizer of $x$ in $G$, $\Gamma_{H_x}=H_x(F) \cap \Gamma$ and $m^{H_x}_S$ is the $S$-component of the Tamagawa measure over $H_x$ (i.e., $m^{H_x}_S\cdot m^X_{S}=m^G_{S}$). This has been established in the following situation:

($*$) $G$ is an almost $F$-simple simply connected group, $G(F_S)$ is not compact and $X$ is symmetric ($\car (F) \neq 2$), see \cite[Theorem~1.1, Example~8.4, 8.5]{Benoist-Oh12}.

\

Theorem \ref{th1} follows from the following more general result.
\begin{thm}\label{th2}
Suppose $X=H\backslash G$ where $G$ is semi-simple and simply connected and $\calX$
is a separated scheme of finite type over $\calO_S$ whose generic fiber is
 $X$. If $G'(F_S)$ is not compact for any non-trivial simple factor $G'$ of $G$ and $H$ is connected reductive without non-trivial $F$-characters, then under the assumption \eqref{equi-dist} we have
\begin{equation}\label{asymptotic}
    N({\calX}, q^n) \sim r_H\cdot q^{(1-\eta_F)\dim X}\sum_{\xi\in \Br_{1,P}(X,G)} (\prod_{v \notin S}  I_v(\calX{}, \xi)) \cdot I_{S} (X, q^n, \xi)
\end{equation}
as $n\rightarrow \infty$.
\end{thm}
\begin{proof}
As in \S~\ref{subsec-orb}, write (finite) equivalence class decompositions \[
\calX(\calO_S)=\coprod_i(\calX(\calO_S)\cap x_i \St(\calX))
\] with $x_i \in \calX(\calO_S)$,
\[
\calX(\calO_S)\cap x_i\St(\calX)=\coprod_j(y_{ij}G(F)\cap x_i\St(\calX))=\coprod_j(y_{ij}G(F)\cap y_{ij}\St(\calX))
\] with $y_{ij}\in \calX(\calO_S)\cap x_i\St(\calX)$ for each $i$, and \[
y_{ij}G(F)\cap x_i\St(\calX)=\coprod_kz^{ij}_k\Gamma
\] with $z^{ij}_k \in y_{ij}G(F)\cap x_i\St(\calX)$ for each $i,j$. Then 
\[
N({\calX}, q^n) \sim \sum_i \sum_j \sum_k \frac{m^{ijk}_{S}( \Gamma_{ijk}\backslash H_{ijk}(F_S)) }{m^G_{S} (\Gamma \backslash G(F_S)) } m^X_{S} (z_k^{ij} \cdot G(F_{S}) \cap X(F_S, q^n ))
\] as $n \rightarrow \infty$ by the assumption \eqref{equi-dist}, where $H_{ijk}$ is the stabilizer of $z_k^{ij}$ in $G$, $\Gamma_{ijk}= H_{ijk}(F) \cap \Gamma$ and $m^{ijk}_{S}$ is the $S$-component of the Tamagawa measure on $H_{ijk}(\bfA_F)$.

By construction there exists some $g_k^{ij}\in G(F)$ such that $z_k^{ij}=y_{ij} \cdot g_k^{ij}$ for each $i,j,k$, then one has $$H_{ijk}=(g_k^{ij})^{-1}H_{ij}g_k^{ij}, $$where $H_{ij}$ is the stabilizer of $y_{ij}$ in $G$. This implies that
\[
m^{ijk}_{S}( \Gamma_{ijk}\backslash H_{ijk}(F_S))= m^{ij}_{S}( g_k^{ij}\Gamma (g_k^{ij})^{-1} \cap H_{ij}(F) \backslash H_{ij}(F_S)),
\] where $m^{ij}_{S}$ is the $S$-component of Tamagawa measure $m^{ij}$ on $H_{ij}(\bfA_F)$.

On the other hand, there also exists $s_k^{ij}\in \St({\calX})$ such that $z_{k}^{ij}=y_{ij} \cdot s_k^{ij}$, thus one obtains\[
\aligned & m^{ij}_{S}( g_k^{ij}\Gamma (g_k^{ij})^{-1} \cap H_{ij}(F) \backslash H_{ij}(F_S)) \\
= &
 q^{(\eta_F-1)\dim H_{ij}}\cdot r_{ij}\cdot m^{ij} (H_{ij}(F)\backslash H_{ij}(F) (H_{ij}(\bfA_F) \cap g_k^{ij} \St({\calX}) (g_k^{ij})^{-1})) \\ &\cdot \prod_{v\notin S} m^{ij}_v (H_{ij}(F_v) \cap g_k^{ij} \St({\calX})_v (g_k^{ij})^{-1})^{-1} \\
 = &
 q^{(\eta_F-1)\dim H}\cdot r_H\cdot m^{ij} (H_{ij}(F)\backslash H_{ij}(F) h_k^{ij} (H_{ij}(\bfA_F) \cap \St({\calX}))) \cdot \prod_{v \notin S} m^{ij}_v (H_{ij}(F_v) \cap  \St({\calX})_v)^{-1}
\endaligned
\] where $h_k^{ij}\coloneqq g_k^{ij}\cdot (s_k^{ij})^{-1}\in H_{ij}(\bfA_F)$ and $r_{ij}=r_H$ since $H_{ij}$ is an inner form of $H$ (see Remark~\ref{rmk-inner form}). Since $G(\bfA_F)=G(F)\St(\calX)$ by the strong approximation property of $G$, we have \[
\coprod_{k}H_{ij}(F) h_k^{ij} (H_{ij}(\bfA_F) \cap \St({\calX}))=H_{ij}(\bfA_F)
\] by Proposition \ref{refine}.  Therefore
\[
\begin{aligned}
    N({\calX}, q^n) \sim &q^{(\eta_F-1)\dim H} \sum_i \frac{r_H\cdot m^X_{S} (x_i \cdot G(F_{S}) \cap X(F_S,  q^n)) }{m^G_{S} (\Gamma \backslash G(F_S)) } \\ &\cdot ( \sum_j  \tau(H_{ij}) \prod_{v \notin S} m^{ij}_v (H_{ij}(F_v) \cap  \St({\calX})_v)^{-1} )
\end{aligned}
\]
as $n\rightarrow \infty$, where $\tau(H_{ij})$ is the Tamagawa number of $H_{ij}$ (note that $H_{ij}$ has no non-trivial $F$-characters as it is an inner form of $H$, see Remark~\ref{rmk-inner form}).

Since \[
m^{ij}_v (H_{ij}(F_v) \cap  \St({\calX})_v)^{-1}=\frac{m^X_v(y_{ij}\St({\calX})_v)}{m^G_{v}(\St({\calX})_v)}=\frac{m^X_v(x_{i}\St({\calX})_v)}{m^G_{v}(\St({\calX})_v)}
\]
and $\tau(H_{ij})=\tau(H)$ by \cite[Theorem~1.4]{Rosengarten21a}, we get that
\[
N({\calX}, q^n) \sim q^{(\eta_F-1)\dim H} \# \Pic(H)\sum_i\frac{r_H\cdot m^X_S(x_i\cdot G(F_S)\cap X(F_S,q^n))\prod_{v \notin S} m^X_v(x_i\St({\calX})_v)}{m^G_{S} (\Gamma \backslash G(F_S)) \prod_{v \notin S} m^G_v(\St({\calX})_v) }
\]
as $n\rightarrow \infty$ by Proposition \ref{sha}, \cite[Theorem~1.1]{Rosengarten21a} and \cite[Remark 1.4]{Rosengarten21b}.
By the strong approximation property for $G$, one has 
\[
q^{(1-\eta_F)\dim G}m^G_{S} (\Gamma \backslash G(F_S)) \prod_{v \notin S} m^G_v(\St({\calX})_v)=m^{G}(G(F)\backslash G(F)\St({\calX}))=m^{G}(G(F)\backslash G(\bfA_F))=1,
\]where the last equality is due to work of Gaitsgory and Lurie (see \cite{Gaitsgory-Lurie19}).
Then we have
\[
N({\calX}, q^n) \sim\#\Pic(H)\sum_i{q^{(1-\eta_F)\dim X}\cdot r_H\cdot m^X_S(x_i\cdot G(F_S)\cap X(F_S,q^n))\prod_{v \notin S} m^X_v(x_i\St({\calX})_v)}
\] as $n\rightarrow \infty$.
Hence
\[
N({\calX}, q^n) \sim\#\Pic(H) \int_{(\prod_{v \notin S}{\calX} ({\calO}_v) \times X(F_S, q^n))^{ \Br_{1,P}(X,G)}} \d m^{X}
\] as $n\rightarrow \infty$ by Proposition \ref{equiv}. The result then follows from Lemma \ref{3.1} and \ref{3.2}.
\end{proof}

\section{Application}\label{sec-example}



In this section, we apply our main result to study the number of integral representations by a sum of squares.

For simplicity, we assume throughout this section that $F=\bF_q(t)$, and $S=\{\infty\}$ consists just of the infinite place given by the degree valuation. In particular, $\calO_S=\bF_q[t]$ and $F_S=F_{\infty}=\bF_q\laurent{t^{-1}}$. For $v\in\Omega_F$, we use $p_v$ to denote a fixed uniformizer of the complete local ring $\calO_v$, where it is understood that $p_v$ is the corresponding monic irreducible polynomial if $v\neq \infty$ and $p_v=t^{-1}$ for $v=\infty$.

Our object of study will be the asymptotic formula for counting integral points on the $\calO_S$-scheme $\calX=\calX_a$ defined by the equation $f(x_1,\ldots,x_r)=a$ in variables $x_1,\ldots,x_r$, where $0\neq a\in\calO_S$ and $f=x_1^2+\cdots+x_r^2$ is a sum of $r\geqslant3$ squares.

The $F$-variety $X=\calX\times_{\calO_S}F$ is a homogeneous space of $G=\Spin(f)$ (cf. \cite[\S~5.1]{CT-Xu09}), which is semi-simple, almost simple and simply connected. Moreover, we have seen in Lemma~\ref{lem-soq} that $f$ is isotropic over $F$ (a fortiori over $F_S$), hence $G(F_S)$ is non-compact by, e.g., \cite{Prasad82} and \cite[\S~23.4]{Borel91}. One also sees by \cite[Theorem~I.3.4(3)]{Lam05} that $X$ always admits an $F$-rational point, regardless of the choice of $a$. For a fixed $a$, as usual a rational point $P$ gives rise to an isomorphism $X\simeq H\backslash G$, with $H\simeq \Spin(h)$ for some quadratic subform $h$ of $f$ of rank $(r-1)$ if $r\geqslant4$, and $H\simeq \mathrm{R}_{K/F}^1\mathbb{G}_m$ the norm-one torus associated to $K\coloneqq F[x]/(x^2-d)$ if $r=3$, where $d=-d(f)a=-a$ (see \cite[\S\S~5.3 and 5.6]{CT-Xu09} for a detailed discussion of these facts).

In the setting above, the relevant Brauer group $\Br_{1,P}(X,G)$ in Theorem~\ref{th2} is trivial or non-trivial accordingly as $r\geqslant4$ or $r=3$ as we will explain shortly, leading to different asymptotic behaviors of $N(\calX,q^n)$, and so we treat the two cases separately in the following. 

\subsection{The \texorpdfstring{$r\geqslant4$}{r>=4} case}

In this case $H$ is also semi-simple simply connected and $X$ is symmetric (\cite[\S~6.4]{Borovoi-Rudnick95}). As $\Br_{1,P}(X,G)=0$ by Lemma~\ref{3.1}, the asymptotic formula of $N(\calX,q^n)$ is reduced by Theorem~\ref{th2} to the computations in the following proposition.

\begin{prop}\label{localdensity}
The local measures of $X$ are computed as follows:\\
(1). For any $v\neq \infty$ with $\deg(p_v)=d$, one has
    \[
    m_v^{X}(\calX(\calO_v))=\lim_{l\to\infty} \dfrac{\#\ab\{ x\in (\calO_v/p_v^l)^r: f(x)\equiv a\Mod{p_v^l} \}}{q^{ld(r-1)}}.
    \]\\
(2). For $v=\infty$, one has
    \[
    m_v^{X}(X(F_v,q^n))=\lim_{l\to \infty} q^{nr+l}\cdot \vol\ab(\{x\in\calO_v^r: f(x)\equiv at^{-2n}\Mod{p_v^{2n+l}}\}).
    \]
    When $k$ is large enough, one has
    \[
    \vol\ab(\{x\in\calO_v^r: f(x)\equiv \frac{a}{t^{2n}}\Mod{p_v^{2n+l}}\})=\dfrac{\#\{ x\in (\calO_v/p_v^k)^r: f(x)\equiv\tfrac{a}{t^{2n}}\Mod{p_v^{2n+l}} \}}{q^{kr}}.
    \]
    
\end{prop}

\begin{proof}
    The proof is similar to the case over $\mathbb{Q}$ as in \cite[\S~1.8]{Borovoi-Rudnick95}, for the sake of completeness we still provide details here. We also denote by $f\colon \bA_F^r\to \bA_F^1$ the morphism defined by the given quadratic form $f$, then $f$ is smooth on $f^{-1}(U)$ where $U\coloneqq \bA_F^1\backslash \{0\}$. The same construction as in \cite[\S~1.3]{Borovoi-Rudnick95} gives a gauge form $\omega^s$ on the fibre $V_s\coloneqq f^{-1}(s)$ for any $s\in U$, which then induces a local measure $m_v^s$ on $V_s(F_v)$ for each $v\in\Omega_F$, with $m_v^a=m_v^{X}$ as $X=V_a$.

    From the construction of $\omega^s$, it follows that for any compactly supported function $\phi$ on $U(F_v)$ and $\psi$ on $f^{-1}(U(F_v))$ which are locally constant we have, for any $v\in\Omega_F$,
    \begin{equation}\label{fibreintegral}
        \int \phi(f(x))\psi(x)\d x_1\cdots\d x_r=\int \phi(s)\ab(\int_{V_s(F_v)} \psi \d m_v^s) \d s.
    \end{equation}

(1).  First consider $v\neq \infty_F$. Put $B\coloneqq \calO_v^r\subseteq F_v^r$. For $l$ large enough, one has $B_l\coloneqq a+p_v^l\calO_v\subseteq U(F_v)$ since $a\neq0$, and $\vol(B_l)=[\calO_v:p_v^l\calO_v]^{-1}=q^{-ld}$. Take for $\psi$ the characteristic function of the compact open subset $B\cap f^{-1}(U(F_v))\subseteq F_v^r$, and for $\phi$ the characteristic function of $B_l\cap U(F_v)\subseteq F_v$. Then \eqref{fibreintegral} gives
    \[
    q^{-ldr}\#\ab\{ x\in (\calO_v/p_v^l)^r: f(x)\equiv a\Mod{p_v^l} \}=\int_{B_l} m_v^s(V_s(\calO_v)) \d s,
    \] and the desired formula follows after dividing by $\vol(B_l)=q^{-ld}$ and passing to the limit.\\
(2). Now let $v=\infty_F$. In this case put $C\coloneqq \{x\in F_v^r: H_v(x)\leqslant q^n\}$. For $l$ large enough, one still has $B_l\coloneqq a+p_v^l\calO_v\subseteq U(F_v)$, and moreover $\vol(B_l)=q^{-l}$. Take this time for $\psi$ the characteristic function of $C\cap f^{-1}(U(F_v))$, and for $\phi$ again the characteristic function of $B_l\cap U(F_v)$. Then \eqref{fibreintegral} gives
    \begin{align*}
        \int_{B_l} m_v^s(V_s(F_v)\cap C)\d s=&\vol\ab(\{ x\in F_v^r: H_v(x)\leqslant q^n, f(x)\equiv a\Mod{p_v^{l}} \})\\=& q^{nr}\cdot \vol\ab(\{x\in\calO_v^r: f(x)\equiv at^{-2n}\Mod{p_v^{2n+l}}\}).
    \end{align*}
    From this the desired formula follows after again dividing by $\vol(B_l)=q^{-l}$ and passing to the limit.
    
    Finally, it is easily seen that $f(x)\equiv f(x+z)\Mod{p_v^{2n+l}}$ if the coordinates of $z$ are divisible by a sufficiently large power of $p_v$ (more precisely, once $z\in (p_v^{2n+l-\delta}\calO_v)^r$, $\delta$ being the smallest $v$-adic valuation among the coefficients of $f$), and thus \[
    f(x)\equiv at^{-2n}\quad \text{ if and only if } \quad f(x+z)\equiv at^{-2n}\Mod{p_v^{2n+l}},
    \]whence the last assertion.
\end{proof}


\subsection{The \texorpdfstring{$r=3$}{r=3} case}

In this case $H\simeq \mathrm{R}_{K/F}^1\mathbb{G}_m$ with $K=F[x]/(x^2-d)$ and $d=-d(f)a=-a$, as we have noted. In order that the group $H$ has no non-trivial $F$-characters, we further assume in the sequel that $d=-a$ is not a square in $F$, in which case $K$ is a quadratic field extension of $F$, and $\widehat{H}(\overline{F})\simeq \mathbb{Z}[\Gal(K/F)]/\mathbb{Z}(1+\sigma)$ where $\sigma$ denotes the generator of $\Gal(K/F)\simeq \mathbb{Z}/2\mathbb{Z}$. In particular, one sees readily that $\widehat{H}(F)=0$ and hence Theorem~\ref{th2} still applies to $\calX$. The readers are refered to \cite[\S~5.6]{CT-Xu09} for unexplained assertions here.


   



Put $(N,q)$ and $(M,Q)$ to be the (free) quadratic $\calO_S$-lattices of respective rank $1$ and $3$, corresponding to $at^2$ and $f=x_1^2+x_2^2+x_3^2$. More concretely, one lets $M\coloneqq \calO_S e_1\oplus \calO_S e_2\oplus \calO_S e_3$ with $Q(x_1e_1+x_2e_2+x_3e_3)=x_1^2+x_2^2+x_3^2$, and $N\coloneqq \calO_S e$ with $q(te)=at^2$. As we have explained in \S~\ref{subsec-qf}, the scheme $\calX$ can alternatively be interpreted as the scheme of representations of $N$ by $M$. Similarly, the generic fiber $X=\calX\times_{\calO_S}F$ can be seen as the variety of representations of $FN$ by $FM$.

The lattice $M$ is obviously unimodular. Since $x_1^2+x_2^2+x_3^2$ is $F$-isotropic, one can take a maximal isotropic vector $v\in M$, as the ring $\calO_S$ is a PID (cf. \cite[\S~81B]{OMeara00}). Then one finds a vector $w\in M$ with $B(v,w)=1/2$ by \cite[82:17]{OMeara00}, where $B$ denotes the associated bilinear form of $Q$. Replacing $w$ by $-Q(w)v+w$ if necessary, one may assume $Q(w)=0$ so that $J\coloneqq \calO_S v+\calO_S w\simeq \begin{pmatrix}
0 & 1/2 \\
1/2 & 0
\end{pmatrix}$. An application of Proposition~\ref{prop-split} then gives a splitting $M=J\perp J'$, necessarily with $J'\simeq \calO_S u$ such that $Q(u)=-1$ since $dM=dJ\cdot dJ'$.

In other words, there exists a linear change of variables $(y_1,y_2,y_3)=(x_1,x_2,x_3)A$ with $A\in \GL_3(\calO_S)$ such that $x_1^2+x_2^2+x_3^2=y_1y_2-y^2_3$, which also amounts to say that \begin{equation}\label{transition}
    A^{-1}(A^{t})^{-1}=\begin{pmatrix}
0 & 1/2 & 0\\
1/2 & 0 & 0\\
0 & 0 & -1
\end{pmatrix}\eqqcolon P.
\end{equation} Since $A$ has its entries in the polynomial ring $\calO_S=\mathbb{F}_q[t]$, we may assume actually $A\in \GL_3(\mathbb{F}_q)$, replacing each entry of $A$ by its constant term if necessary.

\begin{lemma} 
    With $\calX$ as above and $X=\calX\times_{\calO_S}F$, one has
    
    (1). \[\Br_{1}(X,G)/ \Br(F)\simeq \mathbb{Z}/2\] with generator $\xi=(y_2,-a)$, and in particular
    
    (2). \begin{equation}\label{asymptotic-example}
    N({\calX}, q^n) \sim r_H\cdot  \left( \prod_{v \ne \infty}  I_v(\calX, 1)\cdot I_{\infty} (X, q^n, 1) + \prod_{v \ne \infty}  I_v(\calX, \xi) \cdot I_{\infty} (X, q^n, \xi)\right)   
    \end{equation} 
as $n\to \infty$.
\end{lemma}

\begin{proof}
Since $\widehat{H}(\overline{F})\simeq \mathbb{Z}[\Gal(K/F)]/\mathbb{Z}(1+\sigma)$ with $\sigma$ denoting the generator of $\Gal(K/F)$, one has \[
\Pic(H)\simeq H^1(F,\widehat{H})\simeq \mathbb{Z}/2
\] by \cite[Lemma~6.9 (ii)]{Sansuc81} and \cite[Example 3.2.9]{Gille-Szamuely17}. Hence \[\Br_1(X, G)/\Br(F)\simeq \Br_{1,P}(X, G)\simeq \mathbb{Z}/2\] by Lemma~\ref{3.1}.

By the equation $x_1^2+x_2^2+x_3^2=y_1y_2-y^2_3=a$ and a similar argument as in \cite[\S~5.8]{CT-Xu09}, one obtains that $\xi=(y_2, -a) \in \Br_1(X,G)$ and $\xi \notin \Br(F)$, and thus serves as a generator of $\Br_{1}(X,G)/ \Br(F)$.

Finally by class field theory, the group $\Br_{1,P}(X, G)$ in Proposition~\ref{equiv} (and hence in Theorem~\ref{th2}) can be replaced by any set of complete representatives of the quotient $\Br_1(X, G)/\Br(F)$. Hence \eqref{asymptotic-example} follows immediately from Theorem~\ref{th2}.
\end{proof}
Next we would like to take a closer look at the contribution from the $\xi$-summand in \eqref{asymptotic-example}. It turns out that different phenomenon may occur, depending on whether the polynomial $a$ admits a prime factor of odd power in its factorization or not.

\subsubsection*{When $\ord_v(a)$ is odd for some finite place $v$.}
In this case we show that the $\xi$-summand in \eqref{asymptotic-example} necessarily vanishes, and therefore does not contribute to our counting problem. Note that the assumption that $-a$ being a non-square in $F$ is automatic in this case.

\begin{prop}\label{prop-oddlocint}
    Suppose $\ord_v(a)$ is odd for some $v\neq\infty$, then $I_v(\calX,\xi)=0$. In particular, \[
    N({\calX}, q^n) \sim r_H\cdot   \prod_{v \ne \infty}  I_v(\calX, 1)\cdot I_{\infty} (X, q^n, 1)
    \] as $n\to\infty$.
\end{prop}

\begin{proof}
Recall that $\calX$ can be viewed as the scheme of representations of the quadratic lattice $N$ by $M$, and similarly $X$ corresponds to the variety of representations of $FN$ by $FM$. Denote the $F_v$-quadratic space $FM\otimes_{F}F_v$ by $V$, the action of $\rO^+(f)(F_v)=\rO^+(V)$ on $X(F_v)$ is given by $\varphi\cdot \sigma=\sigma^{-1}\circ\varphi$, for $\sigma\in\rO^+(V)$ and $\varphi\colon F_vN\to V$ a representation. Having fixed a rational point $P\in 
X(F)\subseteq X(F_v)$, viewed as a representation $\varphi_P\colon F_vN\to V$, the action map $\pi_P\colon \rO^+(V)\to X(F_v)$ given by $\sigma\mapsto P\cdot \sigma$ is surjective by the Witt extension theorem \cite[42:17]{OMeara00}. Moreover, one checks that $P\cdot\sigma_1=P\cdot \sigma_2$ exactly when $\sigma_1^{-1}$ and $\sigma_2^{-1}$ coincide on $\varphi_P(F_vN)\eqqcolon W$, or in other words, when $\sigma_2\sigma_1^{-1}\in \rO^+(W^\perp)$, where $W^\perp$ is the orthogonal complement of $W$ in $V$. In this way one obtains an $\rO^+(V)$-equivariant bijection \begin{equation}\label{eq-ratpts}
    \rO^+(W^\perp)\backslash\rO^+(V)\simeq X(F_v).
\end{equation} On the integral level, one checks that\[\varphi_P\cdot\sigma=\sigma^{-1}\circ \varphi_P\in \calX(\calO_v)\Leftrightarrow \sigma^{-1}\circ\varphi_P(N_v)\subseteq M_v \Leftrightarrow \varphi_P(N_v)\subseteq \sigma(M_v).\] Hence the inverse image $\pi_P^{-1}(\calX(\calO_v))$ is nothing but the set \[
X^+(M_v/\varphi_P(N_v))\coloneqq \{ \sigma\in\rO^+(V): \varphi_P(N_v)\subseteq\sigma(M_v)\},
\] where $M_v\coloneqq M\otimes_{\calO_S}\calO_v$ and similarly for $N_v$.  
The bijection \eqref{eq-ratpts} thus restricts to an $\rO^+(M_v)$-equivariant bijection \begin{equation}\label{eq-intpts}
    \rO^+(W^\perp)\backslash X^+(M_v/N_v)\simeq \calX(\calO_v).
\end{equation} Here and in the following we abuse notation and simply write $X^+(M_v/N_v)$ for $X^+(M_v/\varphi_P(N_v))$.

Write $\calX(\calO_v)=\calX(\calO_v)_1\coprod \calX(\calO_v)_{-1}$, where the two parts denote the respective subsets on which $\xi$ takes value $1$ and $-1$. Under the identification \eqref{eq-ratpts}, the evaluation map $X(F_v)\to F_v^\times/NF_v(\sqrt{d})^\times\simeq \mathbb{Z}/2$ of $\xi$ coincides with the spinor norm map $\theta\colon \rO^+(V)\to F_v^\times/F_v^{\times2}$ followed by the natural map $F_v^\times/F_v^{\times2}\to F_v^\times/NF_v(\sqrt{d})^\times$, where $NF_v(\sqrt{d})^\times$ stands for the norm group $N_{F_v(\sqrt{d})/F_v}(F_v(\sqrt{d})^\times)$ (see \S~5.8 and the computational recipes after Theorem~6.3 in \cite{CT-Xu09} for explanation). Therefore the two subsets $\calX(\calO_v)_{\pm1}$ correspond under the identification \eqref{eq-intpts} to $\{\sigma\in X^+(M_v/N_v):\theta(\sigma)\in\theta(\rO^+(W^\perp)) \}$ and $\{\sigma\in X^+(M_v/N_v):\theta(\sigma)\notin\theta(\rO^+(W^\perp)) \}$ respectively. 

As $f$ is isotropic over $F$ (a fortiori $F_v$) and $Q(N_v)\calO_v=a\calO_v\subset\calO_v=Q(M_v)\calO_v$, we have \[
\theta(X^+(M_v/N_v))=F_v^\times \ \text{ and }\ \theta(\rO^+(M_v))=\calO_v^\times\cdot F_v^{\times2}
\] by \cite[Proposition~3.7 and Proposition~3.9]{HuLiuXu24}. Moreover, $\theta(\rO^+(W^\perp))$ is the subgroup of $F_v^\times$ consisting of all nonzero elements of the form $\alpha_1\alpha_2$ with $\alpha_i\in Q(W^\perp)$ (\cite[55:2]{OMeara00}). (From this description one can easily deduce that $\theta(\rO^+(W^\perp))=NF_v(\sqrt{d})^\times$, but we omit the details here.) Since $\ord_v(a)$ is odd, $d(W^\perp)=a$ is the class of a uniformizer in $F_v$, it therefore follows from \cite[63:11]{OMeara00} that there exists $\sigma_0\in \rO^+(M_v)$ with $\theta(\sigma_0)\notin \theta(\rO^+(W^\perp))$. Thus the action of $\sigma_0$ on
$\calX(\calO_v)$ interchanges the two parts $\calX(\calO_v)_{\pm1}$ and hence they have the same volume. It therefore follows that $I_v(\calX,\xi)=0$.
\end{proof}

\begin{remark}
\hfill \\
(1). One can also prove the above result via computations using an explicit differential form as in the next paragraph, but we have preferred to give the more conceptual proof here.\\
(2). If $\ord_v(a)>0$ is even, the argument in Proposition~\ref{prop-oddlocint} breaks down since in that case $F_v(\sqrt{d})$ would be an unramified extension of $F_v$, and one then has $\theta(\rO^+(W^\perp))=NF_v(\sqrt{d})^\times=\calO_v^{\times}\cdot F_v^{\times2}=\theta(\rO^+(M_v))$, so there is no way to choose a $\sigma_0\in\rO^+(M_v)$ that interchanges the two parts $\calX(\calO_v)_{\pm1}$. And indeed the integral $I_v(\calX,\xi)$ can be nonzero for suitable choice of $a$ and $v$, as we will see in Example~\ref{ex-even}.

\end{remark}

\subsubsection*{When $\ord_v(a)$ is even for any finite place $v$.}
In this case the contribution from the $\xi$-summand cannot be neglected in general, as we will demonstrate by providing an explicit example. To this end, we need to specify the gauge form that we will work with.

Recall that our $\calO_S$-scheme $\calX$ is given by $x_1^2+x_2^2+x_3^2=a$ and $(y_1,y_2,y_3)=(x_1,x_2,x_3)A$ is a change of variables such that $x_1^2+x_2^2+x_3^2=y_1y_2-y_3^2$. We view the $y_i's$ as global functions on $\calX$.
\begin{lemma}
The differential form \[
\omega\coloneqq-y_1^{-1}\d y_3\wedge \d y_1=y_2^{-1}\d y_3\wedge \d y_2=\frac{1}{2}y_3^{-1}\d y_1\wedge \d y_2
\]
is a $G$-invariant gauge form on $X$.
\end{lemma}
\begin{proof}
More precisely, the three expressions define respectively a (nonzero) differential form on the principal open sets $D(y_1),D(y_2)$ and $D(y_3)$ of $X$. That they coincide with each other on corresponding intersections follows from the identity $y_1\d y_2+y_2\d y_1-2y_3\d y_3=0$, obtained by differentiating the equation $y_1y_2-y^2_3=a$. Therefore they glue to a global nonvanishing differential form $\omega$ (of degree $2=\dim(X)$) on $X$.

It remains to show that $\omega$ is invariant under $G=\Spin(f)$. Since $G$ acts on $X$ through the special orthogonal group $\rO^+(f)$ of $f$, it will be sufficient to prove that $\omega$ is $\rO^+(f)$-invariant.

Note that \begin{align*}
    y_1\d y_2\wedge\d y_3+y_2\d y_3\wedge\d y_1+y_3\d y_1\wedge\d y_2&=(\frac{-y_1y_2}{2y_3}-\frac{y_1y_2}{2y_3}+y_3)\d y_1\wedge\d y_2\\&=-ay_3^{-1}\d y_1\wedge\d y_2=-2a\omega.
\end{align*} Thus $\omega$ can also be expressed in the form $-\frac{1}{2a}(y_1\d y_2\wedge\d y_3+y_2\d y_3\wedge\d y_1+y_3\d y_1\wedge\d y_2)$, which is easily checked to be invariant under $\rO^+(g)$, where $g$ represents the quadratic form $y_1y_2-y_3^2$. But $(y_1,y_2,y_3)=(x_1,x_2,x_3)A$, so when an element $B\in\rO^+(f)$ acts on $(x_1,x_2,x_3)$, the vector $(y_1,y_2,y_3)$ is mapped to $(y_1,y_2,y_3)A^{-1}BA$. Since $(A^{-1}BA)P(A^{-1}BA)^t=A^{-1}BB^t(A^{-1})^{t}=A^{-1}(A^{-1})^{t}=P$ by \eqref{transition}, one sees that indeed $A^{-1}BA\in \rO^+(g)$, and so $\omega$ is invariant under the action of $\rO^+(f)$.


\end{proof}
One can therefore use $\omega$ to give a local measure $|\omega|_v$ on $X(F_v)$ for each $v\in \Omega_F$.

\begin{example}\label{ex-even}

Let $a_0$ be a fixed non-square element in $\mathbb{F}_q$, and put
$a=-a_0(t^2-a_0)^{2m}$ with $m\in \mathbb{Z}_{\geqslant1}$. At any finite place $v$ of $F$ other than $v_0$ corresponding to the polynomial $t^2-a_0$, $\xi$ takes constant value on $\calX(\calO_v)$ by \cite[Lemma~9.3]{CT-Xu09}. Hence the local density $I_v(\calX,\xi)$ equals the (nonzero) volume of $\calX(\calO_v)$. On the other hand, $a_0$ becomes a square at $v_0$ so that $d=-a$ is also a square. Thus the quaternion algebra $(y_2,-a)$ representing $\xi$ is split at $v_0$, and one again has $I_{v_0}(\calX,\xi)\neq0$. It remains to compute the integration $I_\infty(X,q^n,\xi)$. Since only the infinite place is involved in the following, we will abbreviate $\ord_{\infty}$ as $\ord$.

Denote by $U$ the open subset of $X$ defined by $y_2\ne 0$. Then 
 \[
 U\simeq\Spec F[x_1,x_2,x_3,y_2^{-1}]/(x_1^2+x_2^2+x_3^2-a)\simeq\Spec F[y_1,y_2,y_2^{-1},y_3]/(y_1y_2-y_3^2-a).
 \]
Let $\Phi$ be the morphism
\[
\mathbb A_F^1\times \mathbb G_{m,F}=\Spec F[z_1,z_2,z_2^{-1}]\longrightarrow U, \qquad y_3\mapsto z_1,\quad y_2\mapsto z_2,
\]
which is an isomorphism of $F$-schemes. And by abuse of notation we also use $\Phi$ to denote the homeomorphism $F_\infty\times F_\infty^\times\to U(F_\infty).$
Since $-a$ is a non-square element in $F_\infty$, one has $X(F_\infty)=U(F_\infty)$. In particular, for any compact set $C\subseteq X(F_\infty)$ and any locally constant function $j$ on $X(F_\infty)$, we have 
\begin{align*}
  \int_{\substack{C}}j\d |\omega|_\infty&=\int_{\substack{\Phi^{-1}(C)\\}}(j\circ\Phi)\d|\Phi^*(\omega)|_\infty\\&=\int_{\substack{\Phi^{-1}(C)\\}}(j\circ\Phi)\frac{\d z_1\d z_2}{|z_2|_\infty}. 
\end{align*}
   
Note that the value of $\xi$ at a point $x=(x_1,x_2,x_3)\in X(F_\infty)$ is determined by the Hilbert symbol of $x^*\xi=(y_2,-a)=(y_2,a_0)\in \Br(F_\infty)$, namely 
\begin{align*} (-1)^{\ord(y_2)\ord(a_0)}\overline{y_2^{\ord(a_0)}a_0^{-\ord(y_2)}}=a_0^{-\ord(y_2)} \in \mathbb{F}_q^\times/\mathbb{F}_q^{\times2}\simeq \mathbb{Z}/2=\{\pm1\}.
\end{align*}By the choice of $a_0$, we see that this value is just $(-1)^{\ord(y_2)}$. In order to compute the integration $I_\infty(X,q^n,\xi)=\int_{X(F_\infty,q^n)}\xi\d |\omega|_\infty$, we will divide the domain
\begin{align*}
    X(F_\infty,q^n) &=\{ (x_1,x_2,x_3)\in X(F_\infty): |x_i|_{\infty}\leqslant q^n, i=1,2,3 \} \\
    &=\{ (x_1,x_2,x_3)\in X(F_\infty): |y_i|_{\infty}\leqslant q^n, i=1,2,3 \}\\
    &=\{ (x_1,x_2,x_3)\in X(F_\infty): \ord(y_i)\geqslant -n, i=1,2,3 \}
\end{align*}into smaller parts, where the second equality follows from the fact that the transition matrix $A$ from $(x_1,x_2,x_3)$ to $(y_1,y_2,y_3)$ has entries in $\mathbb{F}_q$.

Since we are only concerned with asymptotic behaviors, there is no harm in assuming that $n$ is large enough so that $\ord(a)>-2n$. 

When $\ord(y_3^2)\geqslant \ord(a)$, one has $\ord(y_1)+\ord(y_2)=\ord(a)$ by the identity $y_1y_2-y^2_3=a$. This is obvious if $\ord(y_3^2)> \ord(a)$; when $\ord(y_3^2)= \ord(a)$ it follows from the assumption that $a_0\notin \mathbb{F}_q^{\times2}$ (recall that $a=-a_0(t^2-a_0)^{2m}$). Therefore, in order that $\ord(y_1)\geqslant -n$ and $\ord(y_2)\geqslant -n$ both hold, there are only finitely many possibilities for $\ord(y_2)$. More precisely, $-n\leqslant \ord(y_2)\leqslant n+\ord(a)$. Now
\begin{align}\label{eq-xi1}
    \int_{\substack{X(F_\infty,q^n)\\ \ord(y_3^2)\geqslant \ord(a)}}\xi\d |\omega|_\infty\nonumber&=\int_{\substack{U(F_\infty,q^n)\\ \ord(y_3^2)\geqslant \ord(a)}}(-1)^{\ord(y_2)}\d|\frac{\d y_3\wedge \d y_2}{y_2}|_\infty \\&=\int_{\substack{F_\infty\times F_\infty^\times\\ \ord(z_1^2)\geqslant \ord(a)\\-n\leqslant \ord(z_2)\leqslant n+\ord(a) }} (-1)^{\ord(z_2)}\frac{\d z_1\d z_2}{|z_2|_\infty} \nonumber\\ &=\vol(\{ z_1\in F_\infty: \ord(z_1^2)\geqslant \ord(a) \}) \cdot \sum_{i=-n}^{n+\ord(a)} (-1)^i(1-q^{-1})\nonumber\\
    &=q^{2m}\cdot\sum_{i=-n}^{n+\ord(a)} (-1)^i(1-q^{-1}).
\end{align} As the last summation is either $0$ or $\pm (q-1)/q$ depending on the parities of $n$ and $\ord(a)$, one sees that this part of the integration is bounded in any case.

When $-2n\leqslant\ord(y_3^2)< \ord(a)=-4m$, one has $\ord(y_1)=2\ord(y_3)-\ord(y_2)$ by the identity $y_1y_2-y^2_3=a$. Therefore, if we write $\ord(y_3)=-n+k$ with $0\leqslant k< -2m+n$, then $-n\leqslant \ord(y_2)\leqslant -n+2k$. One then obtains
\begin{align}\label{eq-xi2}
    \int_{\substack{X(F_\infty,q^n)\\ -2n\leqslant\ord(y_3^2)< \ord(a)}}\xi\d |\omega|_\infty&=\int_{\substack{U(F_\infty,q^n)\nonumber\\ -2n\leqslant\ord(y_3^2)<\ord(a)}}(-1)^{\ord(y_2)}\d|\frac{\d y_3\wedge \d y_2}{y_2}|_\infty \\&=\sum^{-2m+n-1}_{k=0}\int_{\substack{F_\infty\times F_\infty^\times\\ \ord(z_1)=-n+k\\-n\leqslant \ord(z_2)\leqslant -n+2k }} (-1)^{\ord(z_2)}\frac{\d z_1\d z_2}{|z_2|_\infty} \nonumber\\ &=\sum^{-2m+n-1}_{k=0}(q-1)q^{-(-n+k+1)}\sum_{i=0}^{2k}(-1)^{-n+i}q^{-n+i}(q-1)q^{-(-n+i+1)} \nonumber\\ &=\sum^{-2m+n-1}_{k=0} (-1)^n (q-1)^2q^{n-k-2}.
\end{align}

Combining \eqref{eq-xi1} and \eqref{eq-xi2}, one concludes easily that $I_{\infty} (X, q^n, \xi) \sim (-1)^n(q-1)q^{n-1} \sim (-q)^n$ as $n\to \infty$. On the other hand, for the trivial Brauer element one computes in the same way that 
\begin{align}\label{eq-11}
    \int_{\substack{X(F_\infty,q^n)\\ \ord(y_3^2)\geqslant \ord(a)}}\d |\omega|_\infty&=\int_{\substack{F_\infty\times F_\infty^\times\\ \ord(z_1^2)\geqslant \ord(a)\\-n\leqslant \ord(z_2)\leqslant n+\ord(a) }} \frac{\d z_1\d z_2}{|z_2|_\infty}\nonumber\\&=q^{2m} \cdot \sum_{i=-n}^{n+\ord(a)}(1-q^{-1})=q^{2m-1}(q-1)(2n+1-4m)
\end{align}
which is linear in $n$, and 
\begin{align}\label{eq-12}
    \int_{\substack{X(F_\infty,q^n)\\-2n\leqslant \ord(y_3^2)< \ord(a)}}\d |\omega|_\infty &=\int_{\substack{U(F_\infty,q^n)\\ -2n\leqslant\ord(y_3^2)<\ord(a)}}\d|\frac{\d y_3\wedge \d y_2}{y_2}|_\infty \nonumber\\&=\sum^{-2m+n-1}_{k=0}\int_{\substack{F_\infty\times F_\infty^\times\\ \ord(z_1)=-n+k\\-n\leqslant \ord(z_2)\leqslant -n+2k }}\frac{\d z_1\d z_2}{|z_2|_\infty} \nonumber\\ &=\sum^{-2m+n-1}_{k=0}(q-1)q^{-(-n+k+1)}\sum_{i=0}^{2k}q^{-n+i}(q-1)q^{-(-n+i+1)} \nonumber\\ &=\sum^{-2m+n-1}_{k=0} (2k+1) (q-1)^2q^{n-k-2} .
\end{align}
Hence $I_{\infty} (X, q^n, 1)\sim  (q+1)q^{n-1} \sim q^n$ as $n\to \infty$ by \eqref{eq-11} and \eqref{eq-12}. One thus sees in this example that the two summands in \eqref{asymptotic-example} have the same order when $n\to \infty$. In other words, the $\xi$-summand contributes non-trivially to the asymptotic formula.
\qed
\end{example}

\bibliographystyle{alpha}
\bibliography{revised_ref}

\

Sheng Chen

\medskip

School of Mathematics and Statistics

Changchun University of Science and Technology


Changchun 130022, China


Email: chenshen1991@cust.edu.cn

\

Jing Liu

\medskip

School of Sciences

Great Bay University


Dongguan 523000, China


Email: liuj@gbu.edu.cn

\

\end{document}